\documentclass[12pt,a4paper,reqno]{amsart}

\usepackage{amsmath,amsthm}
\usepackage{amssymb}

\theoremstyle{theorem}
\newtheorem{theorem}{Theorem}

\theoremstyle{definition}
\newtheorem{definition}{Definition}
\newtheorem{remark}{Remark}
\newtheorem{corollary}{Corollary}

\theoremstyle{remark}

\usepackage{url}
\usepackage[colorlinks,linkcolor=blue,citecolor=blue]{hyperref}

\title{Daehee, hyperharmonic, and power sum polynomials}

\author[J.L. Cereceda]{Jos\'e Luis Cereceda}
\address{Collado Villalba, 28400 -- Madrid, Spain}
\email{jl.cereceda@movistar.es}

\begin{document}

\begin{abstract}
In this paper we consider the Daehee numbers and polynomials of the first and second kind, and give several explicit representations for them. In particular, we express the Daehee polynomials as the derivative of a generalized binomial coefficient. This is done by performing the Stirling transform of the power sum polynomial $S_k(x)$ associated with the sum of $k$th powers of the first $n$ positive integers $S_k(n) = 1^k + 2^k + \cdots + n^k$. Furthermore, we show the relationship between the Daehee polynomials and the hyperharmonic polynomials. This allows us to also express the hyperharmonic polynomials as the derivative of a generalized binomial coefficient. In addition, we reformulate and generalize a number of identities involving the Stirling, Bernoulli, and harmonic numbers, in terms of the Daehee and hyperharmonic polynomials. Finally, in the light of a recent result of Karg{\i}n {\it et al.\/}, we discuss a further extension of the obtained identities in terms of the $r$-Stirling numbers.
\end{abstract}

\maketitle

\vspace{-.5cm}

\section{Introduction}

Before defining the Daehee numbers and polynomials we recall the classical Bernoulli polynomials and Stirling numbers of the first and second kind. The Bernoulli polynomials $B_k(x)$ are defined by the generating function
\begin{equation*}
\frac{t e^{xt}}{e^t -1} = \sum_{k=0}^{\infty} B_k(x) \frac{t^k}{k!}.
\end{equation*}
When $x=0$, $ B_k = B_k(0)$ are the Bernoulli numbers $B_0 =1$, $B_1 = -\frac{1}{2}$, $B_2 =\frac{1}{6}$, $B_3 =0$, $B_4 =-\frac{1}{30}$, etc. An important fact about the Bernoulli numbers is that $B_{2k+1} =0$ for all $k \geq 1$.

For $k,j$ nonnegative integers, we denote by $\overline{\genfrac{[}{]}{0pt}{}{k}{j}} = (-1)^{k-j} \genfrac{[}{]}{0pt}{}{k}{j}$ the signed Stirling numbers of the first kind (with $\genfrac{[}{]}{0pt}{}{k}{j}$ being the unsigned Stirling numbers of the first kind), and we denote by $\overline{\genfrac{\{}{\}}{0pt}{}{k}{j}} = (-1)^{k-j} \genfrac{\{}{\}}{0pt}{}{k}{j}$ the signed Stirling numbers of the second kind (with $\genfrac{\{}{\}}{0pt}{}{k}{j}$ being the unsigned Stirling numbers of the second kind). They are defined through the relations
\begin{equation*}
(x)_k = \prod_{i=0}^{k-1} (x-i) = \sum_{j=0}^{k} \overline{\genfrac{[}{]}{0pt}{}{k}{j}} x^j,
\end{equation*}
where $\overline{\genfrac{[}{]}{0pt}{}{k}{0}} = \delta_{k,0}$, $\overline{\genfrac{[}{]}{0pt}{}{k}{j}} =0$ for $j >k$, and
\begin{equation*}
x^k = \sum_{j=0}^{k} \genfrac{\{}{\}}{0pt}{}{k}{j} (x)_j,
\end{equation*}
where $\genfrac{\{}{\}}{0pt}{}{k}{0} = \delta_{k,0}$, $\genfrac{\{}{\}}{0pt}{}{k}{j} =0$ for $j >k$, and $\delta_{k,0}$ is the Kronecker delta. Alternatively, the (signed) Stirling numbers of the first and the (unsigned) Stirling numbers of the second kind can be defined by the exponential generating function
\begin{align*}
\sum_{n=k}^{\infty} \overline{\genfrac{[}{]}{0pt}{}{n}{k}} \frac{t^n}{n!} & = \frac{1}{k!} \big( \ln{(1+t)} \big)^k,
\intertext{and}
\sum_{n=k}^{\infty} \genfrac{\{}{\}}{0pt}{}{n}{k} \frac{t^n}{n!} & = \frac{1}{k!} \big( e^t -1 \big)^k,
\end{align*}
respectively. Importantly, the Stirling numbers satisfy the orthogonality relations
\begin{equation}\label{ort1}
\sum_{r=0}^{i} \genfrac{\{}{\}}{0pt}{}{i}{r} \overline{\genfrac{[}{]}{0pt}{}{r}{j}} = \delta_{i,j}  \quad\text{and}\quad
\sum_{r=0}^{i} \overline{\genfrac{[}{]}{0pt}{}{i}{r}} \genfrac{\{}{\}}{0pt}{}{r}{j} = \delta_{i,j},
\end{equation}
or, equivalently,
\begin{equation}\label{ort2}
\sum_{r=0}^{i} \overline{\genfrac{\{}{\}}{0pt}{}{i}{r}} \genfrac{[}{]}{0pt}{}{r}{j} = \delta_{i,j}  \quad\text{and}\quad
\sum_{r=0}^{i} \genfrac{[}{]}{0pt}{}{i}{r} \overline{\genfrac{\{}{\}}{0pt}{}{r}{j}} = \delta_{i,j},
\end{equation}
for $i,j \geq 0$, showing the inverse nature of the Stirling numbers of the first and second kind.

Based on the above orthogonality relations, the Stirling transform of sequences of polynomials is defined as follows (see, e.g., \cite{boya} and \cite[Appendix A]{boya2}).
\begin{definition}
Relying on \eqref{ort1}, and given a sequence of polynomials in the real variable $x$, $\{P_j(x) \in \mathbb{Q}[x] \}$, its Stirling transform is the new sequence of polynomials $\{Q_j(x) \in \mathbb{Q}[x] \}$ defined by
\begin{equation*}\label{st1}
Q_k(x) = \sum_{j=0}^{k} \genfrac{\{}{\}}{0pt}{}{k}{j} P_j(x),
\end{equation*}
with inversion
\begin{equation*}\label{st2}
P_k(x) = \sum_{j=0}^{k} \overline{\genfrac{[}{]}{0pt}{}{k}{j}} Q_j(x).
\end{equation*}

Likewise, relying on \eqref{ort2}, we can equally define the Stirling transform of the sequence of polynomials $\{P_j(x) \}$ to be
\begin{equation*}\label{st3}
Q_k(x) = \sum_{j=0}^{k} \overline{\genfrac{\{}{\}}{0pt}{}{k}{j}} P_j(x),
\end{equation*}
with inversion
\begin{equation*}\label{st4}
P_k(x) = \sum_{j=0}^{k} \genfrac{[}{]}{0pt}{}{k}{j} Q_j(x).
\end{equation*}
\end{definition}

Next we define the Daehee polynomials of the first and second kind.
\begin{definition}
Following Kim and Kim \cite[Equation (1.1)]{kim1}, the Daehee polynomials, denoted by $D_k(x)$, are defined by the generating function
\begin{equation}\label{gf1}
\sum_{k=0}^{\infty} D_k(x) \frac{t^k}{k!} = \frac{\ln{(1+t)}}{t} (1+t)^x.
\end{equation}
When $x=0$, $D_k = D_k(0)$ are the Daehee numbers. The generating function \eqref{gf1} is equivalent to having
\begin{equation}\label{def11}
D_k(x) = \sum_{j=0}^{k} \overline{\genfrac{[}{]}{0pt}{}{k}{j}} B_j(x),
\end{equation}
with inverse relation
\begin{equation}\label{def12}
B_k(x) = \sum_{j=0}^{k} \genfrac{\{}{\}}{0pt}{}{k}{j} D_j(x).
\end{equation}
We sometimes refer to $D_k$ and $D_k(x)$ as the Daehee numbers and polynomials of the first kind, to be distinguished from the Daehee numbers and polynomials of the second kind which are defined as follows.
\end{definition}

\begin{definition}
In this paper the Daehee polynomials of the second kind, denoted by $\widehat{D}_k(x)$, are defined by the generating function
\begin{equation}\label{gf2}
\sum_{k=0}^{\infty} \widehat{D}_k(x) \frac{t^k}{k!} = \frac{\ln{(1-t)}}{-t} (1-t)^{1-x}.
\end{equation}
When $x=0$, $\widehat{D}_k = \widehat{D}_k(0)$ are the Daehee numbers of the second kind. The polynomials $\widehat{D}_k(x)$ can be defined equivalently by
\begin{equation}\label{def21}
\widehat{D}_k(x) = \sum_{j=0}^{k} \genfrac{[}{]}{0pt}{}{k}{j} B_j(x),
\end{equation}
with inverse relation
\begin{equation}\label{def22}
B_k(x) = \sum_{j=0}^{k} \overline{\genfrac{\{}{\}}{0pt}{}{k}{j}} \widehat{D}_j(x).
\end{equation}
\end{definition}

Furthermore, from \eqref{gf1} and \eqref{gf2}, it follows that the Daehee polynomials of the first and second kind just defined are related by
\begin{equation}\label{relation}
\widehat{D}_k(x) = (-1)^k D_k(1-x) \quad\text{and}\quad D_k(x) = (-1)^k \widehat{D}_k(1-x).
\end{equation}

\begin{remark}
In \cite[Equation (2.21)]{kim1}, Kim and Kim defined the Daehee polynomials of the second kind by the generating function
\begin{equation*}
\sum_{k=0}^{\infty} \widehat{D}_k(x) \frac{t^k}{k!} = \frac{\ln{(1+t)}}{t} (1+t)^{1-x},
\end{equation*}
which leads to
\begin{equation*}
\widehat{D}_k(x) = (-1)^k \sum_{j=0}^{k} \genfrac{[}{]}{0pt}{}{k}{j} B_j(x).
\end{equation*}
On the other hand, in \cite[Equation (30)]{kim2}, Kim {\it et al.\/} defined the Daehee polynomials of the second kind by the generating function
\begin{equation*}
\sum_{k=0}^{\infty} \widehat{D}_k(x) \frac{t^k}{k!} = \frac{\ln{(1-t)}}{-t} (1-t)^{x+1},
\end{equation*}
which leads to
\begin{equation*}
\widehat{D}_k(x) = \sum_{j=0}^{k} \genfrac{[}{]}{0pt}{}{k}{j} B_j(-x).
\end{equation*}
In this paper we choose the generating function of the Daehee polynomials of the second kind to be \eqref{gf2} so that the sequence $\{ \widehat{D}_j(x) \}$ is related to the sequence of Bernoulli polynomials $\{ B_j(x) \}$ as shown in \eqref{def21} and \eqref{def22}.
\end{remark}

Next we define the well-known concept of generalized binomial coefficient.

\begin{definition}
For real $x$ and for integer $k \geq 0$, the generalized binomial coefficient $\binom{x}{k}$ is defined as
\begin{equation*}
\binom{x}{k} = \frac{ x(x-1)(x-2)\ldots (x-k+1)}{k!}, \quad k \geq 1, \quad\text{and}\quad \binom{x}{0} =1.
\end{equation*}
\end{definition}

The rest of the paper is organized as follows. In Section 2 we express the Daehee polynomials $D_k(x)$ and $\widehat{D}_k(x)$ as the derivative of a generalized binomial coefficient. As we will see, this follows in a straightforward way by performing the Stirling transform of the power sum polynomial $S_k(x)$ associated with the sum of $k$th powers of the first $n$ positive integers $S_k(n) = 1^k + 2^k + \cdots + n^k$. In Section 3 we exhibit the close relationship between the Daehee polynomials and the hyperharmonic polynomials, and express the latter as the derivative of a generalized binomial coefficient. In Section 4 we consider a couple of identities derived by Boyadzhiev \cite[Theorem 6]{boya} involving the Stirling, Bernoulli, and harmonic numbers, and show that each of them is a special case of a more general identity involving the Bernoulli and hyperharmonic polynomials (or, alternatively, the Bernoulli and Daehee polynomials). In Section 5 we look at the so-called {\it harmonic polynomials} introduced by Cheon and El-Mikkawy \cite[Section 5]{cheon} and show how they relate to the hyperharmonic and Daehee polynomials. Finally, in Section 6, in the light of a recent result derived by Karg{\i}n {\it et al.\/} \cite[Theorem 1]{kargin}, we discuss a further extension of the obtained identities (in particular, of the identities \eqref{can1} and \eqref{can2} below) in terms of the $r$-Stirling numbers.

\section{Daehee and power sum polynomials}

Before proving Theorem \ref{th:1} below, let us recall the well-known relationship between the sums of powers of integers $S_k(n) = 1^k + 2^k + \cdots + n^k$ and the Bernoulli polynomials, namely \cite[Equation (15)]{atlan}
\begin{equation}\label{powerber}
S_k(n) = \frac{1}{k+1} \big( B_{k+1}(n+1) - B_{k+1}(1) \big), \quad k \geq 0.
\end{equation}
From \eqref{powerber}, it turns out that $S_k(n)$ is a polynomial in $n$ of degree $k+1$ without constant term. We can naturally extend $S_k(n)$ to a polynomial $S_k(x)$ in which $x$ takes any real value, so that $S_k(x) =1^k + 2^k + \cdots + x^k$ whenever $x$ is a positive integer (with $S_k(0) =0$). In view of \eqref{powerber}, a possible representation of $S_k(x)$ is given by
\begin{equation}\label{powerber2}
S_k(x) = \frac{1}{k+1} \big( B_{k+1}(x+1) - B_{k+1}(1) \big), \quad k \geq 0.
\end{equation}

On the other hand, the derivative of $B_{k+1}(x)$ with respect to $x$ is given by $B_{k+1}^{\prime}(x) = (k+1)B_x(x)$, $k \geq 0$ \cite[Equation (13)]{atlan}. From \eqref{powerber2} we thus have
\begin{equation}\label{der}
S_k^{\prime}(x) = B_{k}(x+1), \quad k \geq 0.
\end{equation}

Next we establish the following result expressing the Daehee polynomials as the derivative of a certain generalized binomial coefficient.
\begin{theorem}\label{th:1}
For $k \geq 0$, the Daehee polynomials of the first and second kind are given by
\begin{align}
D_k(x) & = k! \frac{\text{d}}{\text{d}x} \binom{x}{k+1}, \label{th:11}
\intertext{and}
\widehat{D}_k(x) & = k! \frac{\text{d}}{\text{d}x} \binom{x+k-1}{k+1}, \label{th:12}
\end{align}
respectively.
\end{theorem}
\begin{proof}
To show relation \eqref{th:11}, first we express $S_k(x)$ in terms of the unsigned Stirling numbers of the second kind (see e.g. \cite{shirali})
\begin{equation}\label{var1}
S_k(x) + \delta_{k,0} = \sum_{j=0}^k j! \genfrac{\{}{\}}{0pt}{}{k}{j} \binom{x+1}{j+1}, \quad k \geq 0.
\end{equation}
Noting that $\sum_{j=0}^{k}\overline{\genfrac{[}{]}{0pt}{}{k}{j}} \delta_{j,0} = \delta_{k,0}$, the inverse of the Stirling transform in \eqref{var1} is given by
\begin{equation*}
k! \binom{x+1}{k+1} = \delta_{k,0} + \sum_{j=0}^{k} \overline{\genfrac{[}{]}{0pt}{}{k}{j}} S_j(x).
\end{equation*}
Differentiating both sides of this equation and using \eqref{der} gives
\begin{equation*}
k! \frac{\text{d}}{\text{d}x} \binom{x+1}{k+1} = \sum_{j=0}^{k} \overline{\genfrac{[}{]}{0pt}{}{k}{j}} S_j^{\prime}(x) =
\sum_{j=0}^{k} \overline{\genfrac{[}{]}{0pt}{}{k}{j}} B_j(x+1).
\end{equation*}
Thus, making the transformation $x \to x-1$ and comparing the resulting equation with \eqref{def11}, we get \eqref{th:11}.

To show relation \eqref{th:12}, we start with the following expression for $S_k(x)$ in terms of the signed Stirling numbers of the second kind (see e.g. \cite[Equation (7.5)]{gould})
\begin{equation}\label{var2}
S_k(x) = \sum_{j=0}^k j!  \overline{\genfrac{\{}{\}}{0pt}{}{k}{j}} \binom{x+j}{j+1}, \quad k \geq 0.
\end{equation}
Now, by inversion of the Stirling transform in \eqref{var2}, we have
\begin{equation*}
k! \binom{x+k}{k+1} = \sum_{j=0}^{k} \genfrac{[}{]}{0pt}{}{k}{j} S_j(x).
\end{equation*}
Similarly, differentiating both sides of this equation and using \eqref{der}, we obtain
\begin{equation*}
k! \frac{\text{d}}{\text{d}x} \binom{x+k}{k+1} = \sum_{j=0}^{k} \genfrac{[}{]}{0pt}{}{k}{j} S_j^{\prime}(x) =
\sum_{j=0}^{k} \genfrac{[}{]}{0pt}{}{k}{j} B_j(x+1),
\end{equation*}
and then, replacing $x$ by $x-1$ and comparing the resulting equation with \eqref{def21}, we get \eqref{th:12}.
\end{proof}

A few remarks are in order regarding Theorem \ref{th:1}.
\begin{remark}
By using the relation
\begin{equation}\label{relation2}
\binom{-x}{k} = (-1)^k \binom{x+k-1}{k},
\end{equation}
it is easily seen that the expressions \eqref{th:11} and \eqref{th:12} for $D_k(x)$ and $\widehat{D}_k(x)$ satisfy \eqref{relation}.
\end{remark}

\begin{remark}
When evaluated at $x=0$, the derivatives in \eqref{th:11} and \eqref{th:12} are given by
\begin{equation*}
\left. \frac{\text{d}}{\text{d}x} \binom{x}{k+1}\right|_{x=0} = \frac{(-1)^{k}}{k+1} \quad\text{and}\quad
\left. \frac{\text{d}}{\text{d}x} \binom{x+k-1}{k+1}\right|_{x=0} = -\frac{1}{k(k+1)},
\end{equation*}
and then the Daehee numbers of the first and second kind are given by
\begin{align}
D_k & = (-1)^k \frac{k!}{k+1} = \sum_{j=0}^{k} \overline{\genfrac{[}{]}{0pt}{}{k}{j}} B_j , \quad k \geq 0, \label{dn1}
\intertext{and}
\widehat{D}_k & = -\frac{(k-1)!}{k+1} = \sum_{j=1}^{k} \genfrac{[}{]}{0pt}{}{k}{j} B_j , \quad k \geq 1, \quad
\widehat{D}_0 =1. \label{dn2}
\end{align}
respectively. Furthermore, from \eqref{dn1} and \eqref{dn2}, we obtain the inverse relations
\begin{align}
B_k & = \sum_{j=0}^{k} (-1)^j \frac{j!}{j+1} \genfrac{\{}{\}}{0pt}{}{k}{j} , \quad k \geq 0,  \label{ber1}
\intertext{and}
B_k & = (-1)^{k+1} \sum_{j=1}^{k} (-1)^j \frac{(j-1)!}{j+1} \genfrac{\{}{\}}{0pt}{}{k}{j} , \quad k \geq 1, \quad B_0 =1, \label{ber2}
\end{align}
respectively. Incidentally, it is to be noted that relation \eqref{ber1} is a fairly well-known identity, while relation \eqref{ber2} is seemingly a lesser well-known identity. Identity \eqref{ber2} has recently been derived by a totally different method in \cite[Theorem 1.1, Equation (3)]{jha}.
\end{remark}

\begin{remark}
Searching through the internet for an explicit formula of the derivative of $\binom{x}{n}$ with respect to $x$, we came across with the following result answering a question posted on the Mathematics Stack Exchange website. Specifically, for integer $n \geq 1$, we can learn from \cite[Theorem 1]{stack} that
\begin{equation}\label{stack}
\frac{\text{d}}{\text{d}x} \binom{x}{n} = \sum_{i=1}^{n} \frac{(-1)^{i-1}}{i} \binom{x}{n-i}.
\end{equation}
Hence, combining \eqref{th:11} and \eqref{stack}, we immediately deduce that
\begin{equation}\label{exp11}
D_k(x) = k! \sum_{j=0}^{k} \frac{(-1)^j}{j+1} \binom{x}{k-j}.
\end{equation}
Moreover, since $\widehat{D}_k(x) = (-1)^k D_k(1-x)$, and applying \eqref{relation2}, from \eqref{exp11} we obtain
\begin{equation}\label{exp21}
\widehat{D}_k(x) = k! \sum_{j=0}^{k} \frac{1}{j+1} \binom{x+k-j-2}{k-j}.
\end{equation}
Next we write down the first few instances of the Daehee polynomials of the first and second kind as obtained from \eqref{exp11} and \eqref{exp21}, namely
\begin{align*}
D_0(x) & = 1, \quad D_1(x) = x -\tfrac{1}{2}, \quad D_2(x) = x^2 -2x +\tfrac{2}{3}, \\
D_3(x) & = x^3 -\tfrac{9}{2} x^2 +\tfrac{11}{2}x -\tfrac{3}{2}, \quad D_4(x) = x^4 -8x^3 +21x^2 -20x +\tfrac{24}{5}, \\
D_5(x) & = x^5 -\tfrac{25}{2}x^4 +\tfrac{170}{3}x^3 -\tfrac{225}{2}x^2 + \tfrac{274}{3}x -20, \\
D_6(x) & = x^6 - 18x^5 +125x^4 - 420x^3 + 696x^2 - 504x + \tfrac{720}{7}, \\
\intertext{and}
\widehat{D}_0(x) & = 1, \quad \widehat{D}_1(x) = x -\tfrac{1}{2}, \quad \widehat{D}_2(x) = x^2 -\tfrac{1}{3}, \\
\widehat{D}_3(x) & = x^3 +\tfrac{3}{2} x^2 -\tfrac{1}{2}x -\tfrac{1}{2}, \quad \widehat{D}_4(x) = x^4 +4x^3 +3x^2 -2x -\tfrac{6}{5}, \\
\widehat{D}_5(x) & = x^5 +\tfrac{15}{2}x^4 +\tfrac{50}{3}x^3 +\tfrac{15}{2}x^2 - \tfrac{26}{3}x -4, \\
\widehat{D}_6(x) & = x^6 + 12x^5 + 50x^4 + 80x^3 + 21x^2 - 44x - \tfrac{120}{7}.
\end{align*}

Notice that the expressions in \eqref{exp11} and \eqref{exp21} can alternatively be written as
\begin{align}
D_k(x) & = (-1)^k k! \sum_{j=0}^{k} \frac{(-1)^j}{k+1-j} \binom{x}{j}, \label{exp12} \\
\intertext{and}
\widehat{D}_k(x) & = k! \sum_{j=0}^{k} \frac{1}{k+1-j} \binom{x+j-2}{j}, \label{exp22}
\end{align}
respectively. Therefore, from \eqref{def11} and \eqref{exp12}, we have
\begin{align}
& \sum_{j=0}^{k} \overline{\genfrac{[}{]}{0pt}{}{k}{j}} B_j(x) =  k! \sum_{j=0}^{k} \frac{(-1)^{k-j}}{k+1-j}
\binom{x}{j}, \label{res11} \\
\intertext{with inverse relation}
& B_k(x) = \sum_{j=0}^{k} j! \genfrac{\{}{\}}{0pt}{}{k}{j} \sum_{i=0}^{j} \frac{(-1)^{j-i}}{j+1-i} \binom{x}{i}. \notag
\end{align}
On the other hand, from \eqref{def21} and \eqref{exp22}, we have
\begin{align*}
& \sum_{j=0}^{k} \genfrac{[}{]}{0pt}{}{k}{j} B_j(x) = k! \sum_{j=0}^{k} \frac{1}{k+1-j}
\binom{x+j-2}{j} , \label{res21} \\
\intertext{with inverse relation}
& B_k(x) = \sum_{j=0}^{k} j! \overline{\genfrac{\{}{\}}{0pt}{}{k}{j}} \sum_{i=0}^{j} \frac{1}{j+1-i}
\binom{x+i-2}{i}. \notag
\end{align*}
Incidentally, identity \eqref{res11} was obtained by Todorov \cite[Theorem 5]{todorov}, (expressed in a somewhat different form), and a new proof of it was subsequently given by Boyadzhiev \cite[Theorem 5]{boya}.
\end{remark}

\begin{remark}
Integrating both sides of the expressions \eqref{th:11} and \eqref{th:12} for $D_k(x)$ and $\widehat{D}_k(x)$ yields
\begin{align*}
F_k(x) & = \int_{0}^{x} D_k(t) dt = k! \binom{x}{k+1}, \\
\intertext{and}
\widehat{F}_k(x) & = \int_{0}^{x} \widehat{D}_k(t) dt = k! \binom{x+k-1}{k+1} + \delta_{k,0},
\end{align*}
respectively. In particular, it follows that $F_k(k+1) = k!$ and $\widehat{F}_k(2) = k! + \delta_{k,0}$. For example, using the explicit polynomials $D_5(x)$ and $\widehat{D}_5(x)$ given above, we find
\begin{align*}
& \int_{0}^{6} \big( x^5 -\tfrac{25}{2}x^4 +\tfrac{170}{3}x^3 -\tfrac{225}{2}x^2 + \tfrac{274}{3}x -20 \big) dx = 5!, \\
\intertext{and}
& \int_{0}^{2} \big( x^5 +\tfrac{15}{2}x^4 +\tfrac{50}{3}x^3 +\tfrac{15}{2}x^2 - \tfrac{26}{3}x -4 \big) dx = 5!.
\end{align*}
\end{remark}

\begin{remark}
For integer $r \geq 1$, Kim {\it et al.\/} \cite{kim2} defined the Daehee polynomials of order $r$, $D_k^{(r)}(x)$, by the generating function
\begin{equation*}
\sum_{k=0}^{\infty} D_k^{(r)}(x) \frac{t^k}{k!} = \left(\frac{\ln{(1+t)}}{t}\right)^{r} \! (1+t)^{x}.
\end{equation*}
On the other hand, Liu and Srivastava \cite{liu} defined the N\"{o}rlund polynomials of the second kind $b_k^{(x)}$ as follows
\begin{equation*}
\left( \frac{t}{\ln{(1+t)}} \right)^{x} = \sum_{k=0}^{\infty} b_k^{(x)} t^k.
\end{equation*}
As shown in \cite[Theorem 3]{mustafa}, the Daehee polynomials of order $r$ are related to the N\"{o}rlund polynomials of the second kind by
\begin{equation}\label{order}
D_k^{(r)}(x) = k! \sum_{j=0}^{k} \binom{x}{k-j} b_j^{(-r)},
\end{equation}
for $k \geq 0$ and $r \geq 1$. Furthermore, it turns out that $b_j^{(-1)} = \frac{(-1)^j}{j+1}$ \cite[Remark 4]{liu}, and then, as it should be, the expression \eqref{order} reduces to \eqref{exp11} when $r =1$.
\end{remark}

\section{Daehee and hyperharmonic polynomials}

It is well known that the $r$th harmonic number is the sum of the reciprocals of the first $r$ positive integers
\begin{equation*}
H_r = 1 + \frac{1}{2} + \frac{1}{3} + \cdots + \frac{1}{r},
\end{equation*}
with $H_0 =0$. The generating function of the harmonic numbers is given by
\begin{equation}\label{gfh}
\sum_{r=0}^{\infty} H_r t^r = -\frac{\ln{(1-t)}}{1-t}.
\end{equation}

Conway and Guy \cite[pp. 258--259]{conway} introduced the {\it hyperharmonic numbers}, which are a generalization of the harmonic numbers. Specifically, the $r$th
hyperharmonic number of order $n$, denoted by $H_r^{(n)}$, is defined by
\begin{equation}\label{hhn}
H_r^{(n)} =\left\{
              \begin{array}{ll}
                0 & \text{if}\,\, r \leq 0 \,\, \text{or}\,\, n <0,  \\
                \frac{1}{r}  &  \text{if}\,\, r >0 \,\, \text{and}\,\, n =0,\\
                \sum_{i=1}^r H_i^{(n-1)} &  \text{if}\,\, n,r \geq 1.
              \end{array}
            \right.
\end{equation}
Note that $H_r^{(1)}$ is the ordinary harmonic number $H_r$. The hyperharmonic numbers $H_r^{(n)}$ have the generating function
\begin{equation*}
\sum_{r=0}^{\infty} H_r^{(n)} t^r = -\frac{\ln{(1-t)}}{(1-t)^{n}},
\end{equation*}
which follows from \eqref{gfh} and the definition in \eqref{hhn}.

Of particular interest for our purpose here is the following representation for the hyperharmonic numbers $H_r^{(n)}$ given by Benjamin {\it et al.} \cite[Theorem 1]{benjamin} as a weighted sum of the fractions $\frac{1}{1}, \frac{1}{2}, \ldots, \frac{1}{r}$, namely
\begin{equation}\label{frac}
H_{r}^{(n)} = \sum_{t=1}^{r} \binom{n+r-t-1}{r-t} \frac{1}{t}, \quad n \geq 0, r \geq 1.
\end{equation}

Next we introduce the concept of hyperharmonic polynomials $H_r^{(x)}$, which generalize the hyperharmonic numbers $H_r^{(n)}$.

\begin{definition}
For $x \in R$, the hyperharmonic polynomials $H_r^{(x)} \in \mathbb{Q}[x]$ are defined by the generating function
\begin{equation}\label{gfhp}
\sum_{r=0}^{\infty} H_r^{(x)} t^r = -\frac{\ln{(1-t)}}{(1-t)^{x}},
\end{equation}
with $H_0^{(x)} =0$. By virtue of \eqref{frac} the hyperharmonic polynomials admit the representation
\begin{equation}\label{frac2}
H_{r}^{(x)} = \sum_{t=1}^{r} \binom{x+r-t-1}{r-t} \frac{1}{t},
\end{equation}
from which it follows that $H_{r}^{(x)}$ is a polynomial in $x$ of degree $r-1$ with leading coefficient $\frac{1}{(r-1)!}$ and constant term $\frac{1}{r}$. The first few hyperharmonic polynomials are
\begin{align*}
H_1^{(x)} & = 1, \quad H_2^{(x)} = x +\tfrac{1}{2}, \quad H_3^{(x)} = \tfrac{1}{2}x^2 +x +\tfrac{1}{3}, \\
H_4^{(x)} & = \tfrac{1}{6}x^3 +\tfrac{3}{4} x^2 +\tfrac{11}{12}x +\tfrac{1}{4}, \quad H_5^{(x)} = \tfrac{1}{24}x^4 +\tfrac{1}{3}x^3 +\tfrac{7}{8}x^2 +\tfrac{5}{6}x +\tfrac{1}{5}, \\
H_6^{(x)} & = \tfrac{1}{120}x^5 +\tfrac{5}{48}x^4 +\tfrac{17}{36}x^3 +\tfrac{15}{16}x^2 + \tfrac{137}{180}x +\tfrac{1}{6}, \\
H_7^{(x)} & = \tfrac{1}{720}x^6 + \tfrac{1}{40}x^5 +\tfrac{25}{144}x^4 +\tfrac{7}{12}x^3 + \tfrac{29}{30}x^2 +\tfrac{7}{10}x + \tfrac{1}{7}.
\end{align*}
\end{definition}

Moreover, the hyperharmonic numbers fulfill the property, see \cite[Equation (7)]{benjamin}
\begin{equation*}
H_{r}^{(n)} = \sum_{t=1}^{r} \binom{n+r-t-2}{r-t} H_t, \quad n\geq 0, r \geq 1,
\end{equation*}
and, therefore, we also have that
\begin{equation}\label{frac3}
H_{r}^{(x)} = \sum_{t=1}^{r} \binom{x+r-t-2}{r-t} H_t .
\end{equation}


Now we are ready to prove the following theorem, which establishes the relationship between the hyperharmonic polynomials and the Daehee polynomials of the first and second kinds.
\begin{theorem}\label{th:2}
For integer $k \geq 0$, we have
\begin{align}
D_k(x) & = (-1)^k k! H_{k+1}^{(-x)},  \label{th:21}
\intertext{and\vspace{-1mm}}
\widehat{D}_k(x) & = k! H_{k+1}^{(x-1)}, \label{th:22}
\intertext{where\vspace{-3mm}}
H_{k+1}^{(-x)} & = \sum_{j=0}^{k} \frac{(-1)^j}{k+1-j} \binom{x}{j}, \label{th:23}
\intertext{and\vspace{-3mm}}
H_{k+1}^{(x-1)} & = \sum_{j=0}^{k} \frac{1}{k+1-j} \binom{x+j-2}{j}. \label{th:24}
\end{align}
\end{theorem}
\begin{proof}
First we prove \eqref{th:22}. Then \eqref{th:21} follows from the second relation in \eqref{relation}. Substituting $x \to x-1$ in \eqref{gfhp}, and noting that $H_{0}^{(x-1)} =0$, we have
\begin{equation*}
\sum_{k=0}^{\infty} H_{k+1}^{(x-1)} t^{k+1} = -\frac{\ln{(1-t)}}{(1-t)^{x-1}}.
\end{equation*}
On the other hand, from \eqref{gf2} it follows that
\begin{equation*}
\sum_{k=0}^{\infty} \frac{\widehat{D}_k(x)}{k!} t^{k+1} = -\frac{\ln{(1-t)}}{(1-t)^{x-1}}.
\end{equation*}\nopagebreak
Thus, comparing the left-hand sides of the last two equations gives \eqref{th:22}.

\pagebreak
Moreover, substituting $x \to -x$ and $r \to k+1$ in \eqref{frac2}, and using \eqref{relation2}, we get, after some minor manipulations, \eqref{th:23}. Likewise, \eqref{th:24} is obtained by setting $x \to x-1$ and $r \to k+1$ in \eqref{frac2}.
\end{proof}

By combining \eqref{th:12} and \eqref{th:22} we obtain the following result which gives us the hyperharmonic polynomial $H_{k+1}^{(x)}$ as the derivative of a generalized binomial coefficient.
\begin{corollary}\label{col:1}
For integer $k \geq 0$, we have
\begin{equation}
H_{k+1}^{(x)} = \frac{\text{d}}{\text{d}x} \binom{x+k}{k+1}. \label{col:11}
\end{equation}
\begin{proof}
It follows immediately by setting $x \to x+1$ in \eqref{th:12} and \eqref{th:22}, and then equating the right-hand sides of the resulting equations.
\end{proof}
\end{corollary}

A few remarks are in order regarding Theorem \ref{th:2} and Corollary \ref{col:1}.
\begin{remark}
By combining \eqref{th:21} and \eqref{th:23} [\eqref{th:22} and \eqref{th:24}], we recover \eqref{exp12} [\eqref{exp22}].
\end{remark}

\begin{remark}
Recently, for positive integers $n$ and $r$, Dil and Muniro\u{g}lu, see \cite[Definition 25]{dil}, defined the hyperharmonic numbers of negative order $H_n^{(-r)}$ by
\begin{equation}\label{dil}
H_n^{(-r)} =\left\{
              \begin{array}{ll}
                \dfrac{(-1)^r r!}{n^{\underline{r+1}}}, & n >r \geq 1; \\[4mm]
                \displaystyle\sum_{i=0}^{n-1} (-1)^i \binom{r}{i} \dfrac{1}{n-i}, & r \geq n >1; \\[5mm]
                1, & n=1.
              \end{array}
            \right.
\end{equation}
Hence, for integer $n \geq 1$, we can use \eqref{dil} to properly define $H_{j+1}^{(-n)}$ as follows, see \cite[Equation (22)]{cere}
\begin{equation}\label{dil2}
H_{j+1}^{(-n)} =
\left\{
  \begin{array}{ll}
   \displaystyle\sum_{i=0}^{j} \frac{(-1)^i}{j+1-i} \binom{n}{i}, &  j \geq 1 ; \\
    1, & j=0.
  \end{array}
\right.
\end{equation}
Clearly, the definition \eqref{dil2} is consistent with \eqref{th:23} because the latter equation reduces to \eqref{dil2} when $x \to n$ and $k \to j$. In particular, for $n=1$, we obtain from \eqref{dil2}
\begin{equation}\label{dil3}
H_{k+1}^{(-1)} =
\left\{
  \begin{array}{ll}
   -\dfrac{1}{k(k+1)}, &  k \geq 1 ; \\
    1, & k=0.
  \end{array}
\right.
\end{equation}
\end{remark}

\begin{remark}
By letting $x=0$ in \eqref{th:21} and \eqref{th:22}, we readily obtain the Daehee numbers of the first and second kind, respectively. Indeed, putting $x=0$ in \eqref{th:21} yields
\begin{equation*}
D_k = (-1)^k k! H_{k+1}^{(0)} = (-1)^k \frac{k!}{k+1},
\end{equation*}
in accordance with \eqref{dn1}. Likewise, putting $x=0$ in \eqref{th:22} and using \eqref{dil3} yields
\begin{equation*}
\widehat{D}_k = k! H_{k+1}^{(-1)} =
\left\{
  \begin{array}{ll}
   -\dfrac{(k-1)!}{k+1}, &  k \geq 1 ; \\
    1, & k=0,
  \end{array}
\right.
\end{equation*}
in accordance with \eqref{dn2}.
\end{remark}

\begin{remark}
For integers $n,r \geq 1$, the hyperharmonic numbers fulfill the recurrence relation \cite{benjamin}
\begin{equation*}
H_n^{(r)} = H_{n-1}^{(r)} + H_n^{(r-1)}.
\end{equation*}
As a matter of fact, the above recurrence relation also holds for the hyperharmonic polynomials $H_{k+1}^{(x)}$ since these reduce to $H_{k+1}^{(r)}$ when $x =r$, for each $r =1,2,\ldots\,$. This means that
\begin{equation}\label{rec1}
H_{k+1}^{(x)} = H_k^{(x)} + H_{k+1}^{(x-1)}, \quad k \geq 0.
\end{equation}
As can be easily verified, using \eqref{rec1} in combination with \eqref{th:22} gives the following recurrence relation for the Daehee polynomials of the second kind
\begin{equation}\label{rec2}
\widehat{D}_k(x+1) = \widehat{D}_k(x) + k \widehat{D}_{k-1}(x+1),
\end{equation}
from which we can in turn deduce the following recurrence relation for the Daehee polynomials of the first kind
\begin{equation}\label{rec3}
D_k(1-x) = D_k(-x) + k D_{k-1}(-x).
\end{equation}
Incidentally, this last recurrence was previously derived in \cite[Equation (11)]{mustafa2}.

Furthermore, it is worth noting that, actually, \eqref{rec2} and \eqref{rec3} follow directly by differentiating the well-known Pascal's recurrence for the binomial coefficients, $\binom{k}{j} = \binom{k-1}{j} + \binom{k-1}{j-1} $, and then applying Theorem \ref{th:1}. So, for example, starting from
\begin{equation*}
\binom{x+k}{k+1} = \binom{x+k-1}{k+1} + \binom{x+k-1}{k},
\end{equation*}
we have that
\begin{equation*}
k! \frac{\text{d}}{\text{d}x} \binom{x+k}{k+1} = k! \frac{\text{d}}{\text{d}x} \binom{x+k-1}{k+1} +
k (k-1)! \frac{\text{d}}{\text{d}x} \binom{x+k-1}{k},
\end{equation*}
which, in view of \eqref{th:12}, is just the recurrence \eqref{rec2}.
\end{remark}

\begin{remark}
The power sum polynomial $S_k(x)$ can be expressed in terms of hyperharmonic polynomials as follows. Combining \eqref{def21} and \eqref{th:22}, we get
\begin{equation*}
B_k(x+1) = \sum_{j=0}^{k} \overline{\genfrac{\{}{\}}{0pt}{}{k}{j}} \widehat{D}_j(x+1) =
\sum_{j=0}^{k} \overline{\genfrac{\{}{\}}{0pt}{}{k}{j}} j! H_{j+1}^{(x)}.
\end{equation*}
Thus, from \eqref{powerber2} we obtain
\begin{equation*}
S_k(x) = \frac{1}{k+1} \left[ \sum_{j=0}^{k+1} \overline{\genfrac{\{}{\}}{0pt}{}{k+1}{j}} j! H_{j+1}^{(x)} - B_{k+1}(1)
\right].
\end{equation*}
Now, since $B_{k+1}(1) = (-1)^{k+1} B_{k+1}$, $k\geq 0$ \cite[Equation (8)]{atlan}, and taking into account \eqref{ber1}, we have
\begin{equation}\label{cere1}
S_k(x) = \frac{(-1)^{k+1}}{k+1} \sum_{j=0}^{k+1} (-1)^j j! \genfrac{\{}{\}}{0pt}{}{k+1}{j} \left( H_{j+1}^{(x)}
- \frac{1}{j+1} \right).
\end{equation}
Of course, when we replace $x$ by a positive integer $n \geq 1$, equation \eqref{cere1} becomes
\begin{equation*}
1^k + 2^k + \cdots + n^k = \frac{(-1)^{k+1}}{k+1} \sum_{j=0}^{k+1} (-1)^j j! \genfrac{\{}{\}}{0pt}{}{k+1}{j} \left( H_{j+1}^{(n)}
- \frac{1}{j+1} \right).
\end{equation*}
Incidentally, the last expression for $S_k(n)$ was previously derived in \cite[Theorem 1.1]{cere}.
\end{remark}

\begin{remark}
Notice that, for $x =2$, equation \eqref{th:22} gives $\widehat{D}_k(2) = k! H_{k+1}^{(1)} = k! H_{k+1}$. Therefore, from \eqref{def22} we have
\begin{equation}\label{bin2}
B_k(2) = \sum_{j=0}^{k} j! \overline{\genfrac{\{}{\}}{0pt}{}{k}{j}} H_{j+1}.
\end{equation}
On the other hand, from the difference equation $B_k(x+1)- B_k(x) = k x^{k-1}$ , $k \geq 1$ \cite[Equation (12)]{atlan}, it follows that $B_k(2) = k + (-1)^k B_k$, which holds for all $k \geq 0$, and then from \eqref{bin2} we get the identity (see \cite[Remark 2.3]{cere})
\begin{equation*}
B_k = (-1)^{k+1}k +  \sum_{j=0}^{k} (-1)^j j! \genfrac{\{}{\}}{0pt}{}{k}{j} H_{j+1}, \quad k \geq 0,
\end{equation*}
which involves the Bernoulli and harmonic numbers, as well as the Stirling numbers of the second kind.
\end{remark}

\begin{remark}
For integer $n \geq 0$, it follows from \eqref{col:11} that
\begin{equation}\label{derhn}
H_n^{(x)} = \frac{\text{d}}{\text{d}x} \binom{x+n-1}{n}.
\end{equation}
When $x$ takes on an integer value $r \geq 0$, to get the hyperharmonic number $H_n^{(r)}$ the derivative in \eqref{derhn} should be evaluated at $x =r$, that is,
\begin{equation*}
H_n^{(r)} = \frac{\text{d}}{\text{d}x} \left. \binom{x+n-1}{n} \right|_{x=r},
\end{equation*}
or, equivalently,
\begin{equation}\label{cere2}
H_n^{(r)} = \frac{\text{d}}{\text{d}x} \left. \binom{x+n+r-1}{n} \right|_{x=0}.
\end{equation}
Relation \eqref{cere2} was already established in \cite[Proposition 11]{dil} and \cite[Section 3]{cere2}. For the case that $r =1$, we retrieve the well-known result
\begin{equation*}
H_n = \frac{\text{d}}{\text{d}x} \left. \binom{x+n-1}{n} \right|_{x=1} = \frac{\text{d}}{\text{d}x} \left. \binom{x+n}{n} \right|_{x=0}.
\end{equation*}
\end{remark}

\section{Boyadzhiev's identities}

In \cite{boya} Boyadzhiev derived, among other things, several identities involving the Stirling, Bernoulli, and harmonic numbers. Specifically, in \cite[Theorem 6]{boya}, we find the following two identities (in our notation):
\begin{align}
& B_k = \sum_{j=1}^{k+1} \genfrac{\{}{\}}{0pt}{}{k+1}{j} (-1)^{j-1} (j-1)! H_j, \quad k \geq 0, \label{boya1} \\
& \sum_{j=1}^{k} \overline{\genfrac{[}{]}{0pt}{}{k}{j}} (-1)^j B_{j-1} = (-1)^k \frac{k!}{k^2}, \quad k
\geq 1. \label{boya2}
\end{align}
(We remark that, out of the four identities appearing in \cite[Theorem 6]{boya}, here we only consider explicitly the second and third ones since the first (fourth) identity in \cite[Theorem 6]{boya} is just the inversion of the second (third) one.)

Next we generalize the above identities \eqref{boya1} and \eqref{boya2} by using the machinery we have developed involving the Daehee, hyperharmonic, and power sum polynomials. To this end, we make use of a variant of the formulas for the power sum polynomials in \eqref{var1} and \eqref{var2}, namely (see \cite{cere3}, \cite[Theorem 5]{marques}, \cite[Corollary 2]{dil2})
\begin{align}
S_k(x) & = \sum_{j=0}^{k} j! \genfrac{\{}{\}}{0pt}{}{k+1}{j+1} \binom{x}{j+1}, \quad k \geq 0, \label{var3}
\intertext{and (see \cite[Theorem 3]{dil2})}
S_k(x) + \delta_{k,0} & = \sum_{j=0}^{k} j! \overline{\genfrac{\{}{\}}{0pt}{}{k+1}{j+1}} \binom{x+j+1}{j+1} \quad k \geq 0. \label{var4}
\end{align}

Moreover, we point out that the orthogonality relations in \eqref{ort1} and \eqref{ort2} can also be stated in the form
\begin{align*}
& \sum_{r=0}^{i} \genfrac{\{}{\}}{0pt}{}{i+1}{r+1} \overline{\genfrac{[}{]}{0pt}{}{r+1}{j+1}} = \delta_{i,j},
\quad\quad  \sum_{r=0}^{i} \overline{\genfrac{[}{]}{0pt}{}{i+1}{r+1}} \genfrac{\{}{\}}{0pt}{}{r+1}{j+1} = \delta_{i,j},
\intertext{and}
& \sum_{r=0}^{i} \overline{\genfrac{\{}{\}}{0pt}{}{i+1}{r+1}} \genfrac{[}{]}{0pt}{}{r+1}{j+1} = \delta_{i,j},
\quad\quad  \sum_{r=0}^{i} \genfrac{[}{]}{0pt}{}{i+1}{r+1} \overline{\genfrac{\{}{\}}{0pt}{}{r+1}{j+1}} = \delta_{i,j}.
\end{align*}

Hence, by inversion of the Stirling transform in \eqref{var3}, we have
\begin{equation*}
k! \binom{x}{k+1} = \sum_{j=0}^k \overline{\genfrac{[}{]}{0pt}{}{k+1}{j+1}} S_j(x).
\end{equation*}
Differentiating both sides of this equation and using \eqref{der} gives
\begin{equation*}
k! \frac{\text{d}}{\text{d}x} \binom{x}{k+1} = \sum_{j=0}^{k} \overline{\genfrac{[}{]}{0pt}{}{k+1}{j+1}} S_j^{\prime}(x) =
\sum_{j=0}^{k} \overline{\genfrac{[}{]}{0pt}{}{k+1}{j+1}} B_j(x+1).
\end{equation*}
Therefore, it follows from \eqref{th:11} that
\begin{equation*}
D_k(x-1) = \sum_{j=0}^{k} \overline{\genfrac{[}{]}{0pt}{}{k+1}{j+1}} B_j(x),
\end{equation*}
and then, using \eqref{exp12}, we get
\begin{equation}\label{gen1}
\sum_{j=0}^{k} \overline{\genfrac{[}{]}{0pt}{}{k+1}{j+1}} B_j(x) =
k! \sum_{j=0}^{k} \frac{(-1)^{k-j}}{k+1-j} \binom{x-1}{j},
\end{equation}
which generalizes \eqref{boya2}. Indeed, for the case that $x=1$ identity \eqref{gen1} becomes
\begin{equation*}
\sum_{j=0}^{k} (-1)^j \overline{\genfrac{[}{]}{0pt}{}{k+1}{j+1}} B_j = (-1)^k \frac{k!}{k+1},
\end{equation*}
which can be easily converted into \eqref{boya2}.

On the other hand, the inversion of \eqref{gen1} is given by
\begin{equation}\label{gen2}
B_k(x) = \sum_{j=0}^{k} (-1)^j j! \genfrac{\{}{\}}{0pt}{}{k+1}{j+1} \sum_{i=0}^{j} \frac{(-1)^{i}}{j+1-i} \binom{x-1}{i},
\end{equation}
which generalizes \eqref{boya1}. Indeed, since $\binom{-1}{i} = (-1)^i$, we have that
\begin{equation*}
\sum_{i=0}^{j} \frac{(-1)^{-i}}{j+1-i} \binom{-1}{i} = H_{j+1},
\end{equation*}
and then, for the case that $x=0$ identity \eqref{gen2} becomes
\begin{equation*}
B_k = \sum_{j=0}^{k} \genfrac{\{}{\}}{0pt}{}{k+1}{j+1} (-1)^j j! H_{j+1},
\end{equation*}
which can be easily converted into \eqref{boya1}.

We now turn our attention to the polynomial formula \eqref{var4}. Noting that $\sum_{j=0}^{k}\genfrac{[}{]}{0pt}{}{k+1}{j+1} \delta_{j,0} = \genfrac{[}{]}{0pt}{}{k+1}{1} =k! $, the inverse of the Stirling transform in \eqref{var4} turns out to be
\begin{equation}\label{gazette}
k! \binom{x+k+1}{k+1} = k! + \sum_{j=0}^k \genfrac{[}{]}{0pt}{}{k+1}{j+1} S_j(x).
\end{equation}
Differentiating both sides of this equation and using \eqref{der} gives
\begin{equation*}
k! \frac{\text{d}}{\text{d}x} \binom{x+k+1}{k+1} = \sum_{j=0}^{k} \genfrac{[}{]}{0pt}{}{k+1}{j+1} S_j^{\prime}(x) =
\sum_{j=0}^{k} \genfrac{[}{]}{0pt}{}{k+1}{j+1} B_j(x+1).
\end{equation*}
Consequently, we have from \eqref{th:12}
\begin{equation*}
\widehat{D}_k(x+1) = \sum_{j=0}^{k} \genfrac{[}{]}{0pt}{}{k+1}{j+1} B_j(x),
\end{equation*}
or, equivalently,
\begin{equation}\label{can1}
\sum_{j=0}^{k} \genfrac{[}{]}{0pt}{}{k+1}{j+1} B_j(x) = k! H_{k+1}^{(x)},
\end{equation}
where we have used \eqref{th:22}. Thus, we find from \eqref{th:24} that
\begin{equation*}\label{gen3}
\sum_{j=0}^{k} \genfrac{[}{]}{0pt}{}{k+1}{j+1} B_j(x) = k! \sum_{j=0}^{k} \frac{1}{k+1-j} \binom{x+j-1}{j},
\end{equation*}
with inverse relation
\begin{equation}\label{gen4}
B_k(x) = \sum_{j=0}^{k} j! \overline{\genfrac{\{}{\}}{0pt}{}{k+1}{j+1}} \sum_{i=0}^{j} \frac{1}{j+1-i} \binom{x+i-1}{i}.
\end{equation}
When $x=0$, from \eqref{gen4} we obtain the identity
\begin{equation*}
B_k = (-1)^k \sum_{j=0}^{k} (-1)^j \frac{j!}{j+1} \genfrac{\{}{\}}{0pt}{}{k+1}{j+1},
\end{equation*}
which may be compared with \eqref{ber1}.

\begin{remark}
A proof of equation \eqref{gazette} without relying on \eqref{var4} was given in \cite{cere4} for the case in which $x$ is a natural number $n =1,2,\ldots\,$.
\end{remark}

\begin{remark}
As is well known (see, e.g., \cite{benjamin} and \cite{conway}), for integers $n\geq 0$ and $r\geq 1$, the hyperharmonic and harmonic numbers are related by
\begin{equation*}
H_n^{(r)} = \binom{n+r-1}{r-1} \big( H_{n+r-1} - H_{r-1} \big).
\end{equation*}
Hence, taking $x \to i+1$ in \eqref{can1}, where $i$ is a nonnegative integer, we obtain
\begin{equation*}
\sum_{j=0}^{k} \genfrac{[}{]}{0pt}{}{k+1}{j+1} B_j(i+1) = k! \binom{k+i+1}{i} \big( H_{k+i+1} - H_{i} \big),
\end{equation*}
with inverse relation
\begin{equation*}
\sum_{j=0}^{k} (-1)^j \genfrac{\{}{\}}{0pt}{}{k+1}{j+1} \binom{j+i+1}{i} j! \big( H_{j+i+1} - H_{i} \big) = (-1)^k B_k(i+1),
\end{equation*}
thus retrieving the respective identities (5.7) and (5.10) in \cite{wang}.
\end{remark}

\section{Harmonic and hyperharmonic polynomials}

In \cite[Equation (28)]{cheon}, Cheon and El-Mikkawy defined the harmonic polynomials $H_k(x)$ by the generating function
\begin{equation}\label{gfhap}
\sum_{k=0}^{\infty} H_k(x) t^k = \frac{-\ln (1-t)}{t (1-t)^{1-x}},
\end{equation}
where $H_k(0) = H_{k+1}$. Comparing \eqref{gfhap} with the generating function for the hyperharmonic polynomials in \eqref{gfhp}, we can readily establish the relationship between the harmonic $H_k(x)$ and the hyperharmonic $H_k^{(x)}$ polynomials, namely
\begin{align}
H_k(x) & = H_{k+1}^{(1-x)},  \quad k \geq 0,  \label{hh1}
\intertext{or, equivalently,}
H_k(x) & = \frac{(-1)^k}{k!} D_{k}(x-1),  \quad  k \geq 0.  \label{hh2}
\end{align}
Thus, using \eqref{hh1} and \eqref{th:23}, we obtain the following explicit representation for the harmonic polynomials (see \cite[Remark 5.2]{cere})
\begin{equation}\label{hh3}
H_k(x) = \sum_{j=0}^{k} \frac{(-1)^j}{k+1-j} \binom{x-1}{j}.
\end{equation}
Alternatively, from \eqref{hh1} and \eqref{frac2}, and applying \eqref{relation2}, we get
\begin{equation*}
H_k(x) = \sum_{t=1}^{k+1} \binom{k+1-t-x}{k+1-t} \frac{1}{t},
\end{equation*}
in accordance with \cite[Theorem 5.4]{cheon}. Likewise, from \eqref{hh1} and \eqref{frac3}, and applying \eqref{relation2}, we find that
\begin{equation}\label{cheon}
H_k(x) = \sum_{j=0}^{k} (-1)^{k-j} \binom{x}{k-j} H_{j+1},
\end{equation}
which is just \cite[Equation (33)]{cheon}.

Incidentally, it is worth pointing out that combining \eqref{hh2} and \eqref{cheon} allows us to express the Daehee polynomials of the first kind in terms of the harmonic numbers as follows
\begin{equation}\label{dae0}
D_k(x) = k! \sum_{j=0}^{k} (-1)^j \binom{x+1}{k-j} H_{j+1},
\end{equation}
which may be compared with \eqref{exp11}. Notice that, for $x=0$, equation \eqref{dae0} reads as
\begin{equation}\label{dae1}
D_k = k! \sum_{j=0}^{k} (-1)^j \binom{1}{k-j} H_{j+1} = (-1)^k k! \big( H_{k+1} - H_k \big).
\end{equation}
Actually, \eqref{dae1} is a particular case of the following result obtained in \cite[Theorem 1]{hyper}
\begin{equation*}
D_k = k! \sum_{j=0}^{k} (-1)^j \binom{r}{k-j} H_{j+1}^{(r)},
\end{equation*}
which holds for any integer $r \geq 0$. On the other hand, as a by-product, from \eqref{dae0} and \eqref{th:11} we obtain
\begin{equation*}
\frac{\text{d}}{\text{d}x} \binom{x}{k+1} = \sum_{j=0}^{k} (-1)^j \binom{x+1}{k-j} H_{j+1}.
\end{equation*}

Let us also observe that, by using \eqref{hh3} in \eqref{gen1}, we get
\begin{equation}\label{can2}
\sum_{j=0}^{k} \overline{\genfrac{[}{]}{0pt}{}{k+1}{j+1}} B_j(x) = (-1)^k k! H_k(x),
\end{equation}
thus retrieving \cite[Theorem 5.1]{cheon}, in a slightly different form. In particular, for $x=0$ the last expression reduces to the well-known identity
\begin{equation*}
\sum_{j=0}^{k} (-1)^j \genfrac{[}{]}{0pt}{}{k+1}{j+1} B_j = k! H_{k+1}.
\end{equation*}

Lastly, by means of \eqref{th:11} and \eqref{hh2}, we get yet another formula for the harmonic polynomials, namely
\begin{equation*}
H_k(x) = (-1)^k \frac{\text{d}}{\text{d}x} \binom{x-1}{k+1}.
\end{equation*}
As a simple application of this formula, for $k=4$ we get
\begin{equation*}
H_4(x) = \frac{\text{d}}{\text{d}x} \binom{x-1}{5} = \tfrac{1}{24}x^4 - \tfrac{1}{2}x^3 + \tfrac{17}{8}x^2 -
\tfrac{15}{4}x + \tfrac{137}{60}.
\end{equation*}

\section{Conclusion}

As is well known, see e.g. \cite{broder} and \cite[Chapter 8]{mezo}, for integer $r \geq 0$, the (unsigned) $r$-Stirling numbers of the first and second kind, denoted by $\genfrac{[}{]}{0pt}{}{k+r}{j+r}_r$ and $\genfrac{\{}{\}}{0pt}{}{k+r}{j+r}_r$, respectively, are defined by the exponential generating functions
\begin{align*}
& \sum_{n=k}^{\infty} \genfrac{[}{]}{0pt}{}{n+r}{k+r}_r \, \frac{t^n}{n!} = \frac{1}{k!} \left( \frac{1}{1-t}\right)^{r}
\left[ \ln{\left( \frac{1}{1-x} \right)} \right]^k, 
\intertext{and}
& \sum_{n=k}^{\infty} \genfrac{\{}{\}}{0pt}{}{n+r}{k+r}_r \, \frac{t^n}{n!} = \frac{1}{k!} e^{rt} \big( e^t -1 \big)^k .
\end{align*}
By definition, we have $\genfrac{[}{]}{0pt}{}{n}{k}_0 = \genfrac{[}{]}{0pt}{}{n}{k}$ and $\genfrac{[}{]}{0pt}{}{n+1}{k+1}_1 = \genfrac{[}{]}{0pt}{}{n+1}{k+1}$; similarly, $\genfrac{\{}{\}}{0pt}{}{n}{k}_0 = \genfrac{\{}{\}}{0pt}{}{n}{k}$ and $\genfrac{\{}{\}}{0pt}{}{n+1}{k+1}_1 = \genfrac{\{}{\}}{0pt}{}{n+1}{k+1}$. Also, the signed $r$-Stirling numbers of the first and second kind are defined by $\overline{\genfrac{[}{]}{0pt}{}{n+r}{k+r}}_r = (-1)^{n-k} \genfrac{[}{]}{0pt}{}{n+r}{k+r}_r$ and $\overline{\genfrac{\{}{\}}{0pt}{}{n+r}{k+r}}_r = (-1)^{n-k} \genfrac{\{}{\}}{0pt}{}{n+r}{k+r}_r$. Moreover, the $r$-Stirling numbers satisfy the orthogonality relations
\begin{align*}
& \sum_{s=0}^{i} \genfrac{\{}{\}}{0pt}{}{i+r}{s+r}_r \overline{\genfrac{[}{]}{0pt}{}{s+r}{j+r}_r} = \delta_{i,j},
\quad\quad  \sum_{r=0}^{i} \overline{\genfrac{[}{]}{0pt}{}{i+r}{s+r}_r} \genfrac{\{}{\}}{0pt}{}{s+r}{j+r}_r = \delta_{i,j},
\intertext{or, equivalently,}
& \sum_{s=0}^{i} \overline{\genfrac{\{}{\}}{0pt}{}{i+r}{s+r}_r} \genfrac{[}{]}{0pt}{}{s+r}{j+r}_r = \delta_{i,j},
\quad\quad  \sum_{s=0}^{i} \genfrac{[}{]}{0pt}{}{i+r}{s+r}_r \overline{\genfrac{\{}{\}}{0pt}{}{s+r}{j+r}_r} = \delta_{i,j}.
\end{align*}

Next we show up a simple generalization of the identities \eqref{can1} and \eqref{can2} in terms of the $r$-Stirling numbers. Indeed, it turns out that
\begin{align}
\sum_{j=0}^{k} \genfrac{[}{]}{0pt}{}{k+r}{j+r}_r  B_j(x) & = k! H_{k+1}^{(x+r-1)}, \label{sgen1}
\intertext{and}
\sum_{j=0}^{k} \overline{\genfrac{[}{]}{0pt}{}{k+r}{j+r}}_r  B_j(x) & = (-1)^k k! H_{k}(x-r+1), \label{sgen2}
\end{align}
with inverse relations
\begin{align}
B_k(x) & = \sum_{j=0}^{k} j! \overline{\genfrac{\{}{\}}{0pt}{}{k+r}{j+r}}_r  H_{j+1}^{(x+r-1)}, \label{sgen3}
\intertext{and}
B_k(x) & = \sum_{j=0}^{k} (-1)^j j! \genfrac{\{}{\}}{0pt}{}{k+r}{j+r}_r  H_{j}(x-r+1), \label{sgen4}
\end{align}
where $H_{k}^{(x)}$ and $H_k(x)$  are the hyperharmonic and harmonic polynomials, respectively. Clearly, identity \eqref{sgen1} [\eqref{sgen2}] reduces to \eqref{can1} [\eqref{can2}] when $r=1$. Furthermore, recalling \eqref{th:22} [\eqref{hh2}], identity \eqref{sgen1} [\eqref{sgen2}] reduces to \eqref{def21} [\eqref{def11}] when $r=0$. On the other hand, identity \eqref{sgen3} [\eqref{sgen4}] reduces to \eqref{gen4} [\eqref{gen2}] when $r=1$. Furthermore, recalling \eqref{th:22} [\eqref{hh2}], identity \eqref{sgen3} [\eqref{sgen4}] reduces to \eqref{def22} [\eqref{def12}] when $r=0$.

Identities \eqref{sgen1}--\eqref{sgen4} can be derived by means of the generating functions for the $r$-Stirling numbers, Bernoulli, hyperharmonic, and harmonic polynomials. This having said, however, it should be stressed that identity \eqref{sgen1} (and, as a consequence, its inverse \eqref{sgen3}) arises as a special case (see Remark \ref{remark} below) of the more general identity established by Karg{\i}n {\it et al.\/} in \cite[Theorem 1]{kargin}, namely
\begin{equation}\label{kargin}
\sum_{j=0}^{k} \genfrac{[}{]}{0pt}{}{k+r}{j+r}_r  B_j^{(p)}(x) = k! H_{k+1}^{(p,x+r)}, \quad k,r \geq 0,
\end{equation}
where $B_k^{(p)}(x)$ are the poly-Bernoulli polynomials \cite{bayad}, and $H_n^{(p,r)}$ are the so-called {\it generalized hyperharmonic numbers\/} $H_n^{(p,r)}$ which are defined recursively by \cite{dil3,kargin}
\begin{equation*}
H_n^{(p,r)} = \sum_{k=1}^{n} H_k^{(p,r-1)},
\end{equation*}
with $H_k^{(p,0)} = 1/ k^p$ (for $k \geq 1$), and $H_n^{(p,1)} = \sum_{k=1}^{n} 1/k^p$. Furthermore, for $p=1$, the generalized hyperharmonic numbers $H_k^{(1,r)}$ reduce to the standard hyperharmonic numbers defined in \eqref{hhn}.

We conclude this paper with the following remarks.

\begin{remark}\label{remark}
As noted in \cite{kargin}, by using the relation $B_k^{(1)}(x-1) = B_k(x)$, and taking $p=1$ and $x \to x-1$ in \eqref{kargin}, we get \eqref{sgen1}.
\end{remark}

\begin{remark}
For $x=0$ and $r=1$, the identity \eqref{kargin} becomes
\begin{equation}\label{benyi}
\sum_{j=0}^{k} \genfrac{[}{]}{0pt}{}{k+1}{j+1}  B_j^{(p)} = k! H_{k+1}^{(p,1)},
\end{equation}
where $B_j^{(p)} = B_j^{(p)}(0)$ is the $j$th poly-Bernoulli number with index $p \in \mathbb{Z}$. On the other hand, for integer $k \geq 0$, we have that $S_k(n) = \sum_{i=1}^n i^k = H_n^{(-k,1)}$. Therefore,
using \eqref{benyi}, $S_k(n)$ can be expressed in the form
\begin{equation}\label{benyi2}
S_k(n) = \frac{1}{(n-1)!} \sum_{j=0}^{n-1} \genfrac{[}{]}{0pt}{}{n}{j+1}  B_j^{(-k)},   \quad k \geq 0, n\geq 1,
\end{equation}
where the poly-Bernoulli number with negative index $B_j^{(-k)}$ is given explicitly by \cite{benyi}
\begin{equation*}
B_j^{(-k)} = (-1)^j \sum_{i=0}^{j} (-1)^i i! \genfrac{\{}{\}}{0pt}{}{j}{i} (i+1)^k.
\end{equation*}
Incidentally, formula \eqref{benyi2} has been recently derived in \cite[Corollary 2.18]{benyi2} in a slightly different form; see also \cite[Corollary 2.17]{benyi2} for a generalization of \eqref{benyi2} involving the hypersums of powers of integers \cite{lai}.
\end{remark}

\begin{remark}
Recalling that $B_k(1) = (-1)^k B_k$, and putting $x=1$ in \eqref{sgen3}, we obtain
\begin{equation*}
B_k = \sum_{j=0}^{k} (-1)^j j! \genfrac{\{}{\}}{0pt}{}{k+r}{j+r}_r H_{j+1}^{(r)},
\end{equation*}
which holds for any integers $k,r \geq 0$. This identity also follows by letting $p=1$ and $q+1 =r$ in \cite[Equation (16)]{kargin}. In particular, for $r=0$ we recover \eqref{ber1}.
\end{remark}

\begin{remark}
From \eqref{hh1}, it follows that $H_k(1) = H_{k+1}^{(0)} = \frac{1}{k+1}$. Therefore, by setting $x =r$ in \eqref{sgen4}, for integer $r \geq 0$, we obtain
\begin{equation*}
B_k(r) = \sum_{j=0}^{k} (-1)^j j! \genfrac{\{}{\}}{0pt}{}{k+r}{j+r}_r  H_{j}(1) = \sum_{j=0}^{k} (-1)^j \frac{j!}{j+1} \genfrac{\{}{\}}{0pt}{}{k+r}{j+r}_r ,
\end{equation*}
thus retrieving \cite[Theorem 1]{guo} (see also \cite{boya3}). This expression for $B_k(r)$ can also be obtained by taking $x \to 1-r$ in \eqref{sgen3} and using the property that $B_k(1-x) = (-1)^k B_k(x)$, $k \geq 0$ (\cite[Equation (18)]{atlan}).
\end{remark}

\begin{remark}
From \cite[Theorem 2]{benjamin}, we know that, for $n,r \geq 0$, $n ! H_n^{(r)} = \genfrac{[}{]}{0pt}{}{n+r}{r+1}_r$. Hence, from \eqref{th:22} it readily follows that $(k+1) \widehat{D}_k(r+1) = (k+1)! H_{k+1}^{(r)} = \genfrac{[}{]}{0pt}{}{k+r+1}{r+1}_r$. By \eqref{def22}, this in turn implies that
\begin{equation*}
B_k(r+1) = \sum_{j=0}^{k} \frac{1}{j+1} \overline{\genfrac{\{}{\}}{0pt}{}{k}{j}} \genfrac{[}{]}{0pt}{}{j+r+1}{r+1}_r,
\end{equation*}
which, of course, reduces to \eqref{bin2} when $r=1$.
\end{remark}

\begin{remark}
By using \eqref{cheon} in \eqref{sgen4}, we obtain
\begin{equation*}
B_k(x) = \sum_{j=0}^{k} j! \genfrac{\{}{\}}{0pt}{}{k+r}{j+r}_r \sum_{i=0}^j (-1)^i \binom{x-r+1}{j-i} H_{i+1}.
\end{equation*}
In particular, for $r=1$ we get the following identity expressing the Bernoulli polynomials in terms of the Stirling numbers of the second kind and the harmonic numbers
\begin{equation*}
B_k(x) = \sum_{j=0}^{k} j! \genfrac{\{}{\}}{0pt}{}{k+1}{j+1} \sum_{i=0}^j (-1)^i \binom{x}{j-i} H_{i+1}.
\end{equation*}
\end{remark}


\begin{thebibliography}{99}


\bibitem{bayad} Bayad A., \& Hamahata Y., Polylogarithms and poly-Bernoulli polynomials, {\it Kyushu J. Math.}, 65, 15-24, 2011.

\bibitem{benjamin} Benjamin A.T., Gaebler D., \& Gaebler R., A combinatorial approach to hyperharmonic numbers, {\it Integers}, 3, Article \#A15, 1--9, 2003.

\bibitem{benyi} B\'{e}nyi B., \& Hajnal P., Poly-Bernoulli numbers and Eulerian numbers, {\it Journal of Integer Sequences}, 21, Article 18.6.1, 1--10, 2018.

\bibitem{benyi2} B\'{e}nyi B., \& Matsusaka T., Extensions of the combinatorics of poly-Bernoulli numbers, arXiv:2106.05585v1.

\bibitem{boya} Boyadzhiev K.N., New identities with Stirling, hyperharmonic, and derangement numbers, Bernoulli and Euler polynomials, powers, and factorials, {\it Journal of Combinatorics and Number Theory}, 11(1), 43--58, 2019. 

\bibitem{boya2} Boyadzhiev K.N., {\it Notes on the Binomial Transform}, World Scientific. Singapore, 2018.

\bibitem{boya3} Boyadzhiev K.N., Bernoulli, poly-Bernoulli, and Cauchy polynomials in terms of Stirling and $r$-Stirling numbers, arXiv:1611.02377v2.

\bibitem{broder} Broder A.Z., The $r$-Stirling numbers, {\it Discrete Math.}, 49(3), 241--259, 1984.

\bibitem{cere} Cereceda J.L., Sums of powers of integers and hyperharmonic numbers, {\it Notes on Number Theory and Discrete Mathematics}, 27(2), 101--110, 2021.

\bibitem{cere2} Cereceda J.L., An introduction to hyperharmonic numbers, {\it International Journal of Mathematical Education in Science and Technology}, 46(3), 461--469, 2015.

\bibitem{cere3} Cereceda J.L., Newton's interpolation polynomial for the sums of powers of integers, {\it Amer. Math. Monthly}, 122(10), 1007, 2015.

\bibitem{cere4} Cereceda J.L., An alternative recursive formula for the sums of powers of integers, {\it The Mathematical Gazette}, 100(548), 233--238, 2016.

\bibitem{cheon} Cheon G.-S., \&  El-Mikkawy M.E.A., Generalized harmonic numbers with Riordan arrays, {\it Journal of Number Theory}, 128(2), 413--425, 2008.

\bibitem{marques} Chrysafi L., \& Marques C.A., Sums of powers via matrices, {\it Mathematics Magazine}, 94(1), 43--52, 2021.

\bibitem{conway} Conway J.H., \& Guy R.K., {\it The Book of Numbers}. Copernicus. New York, 1996.

\bibitem{dil} Dil A., \& Muniro\v{g}lu E., Applications of derivative and difference operators on some sequences. {\it Applicable Analysis and Discrete Mathematics}, 14(2), 406--430, 2020.

\bibitem{dil3} Dil A., Mez\H{o} I., \& Cenkci M., Evaluation of Euler-like sums via Hurwitz zeta values, {\it Turkish Journal of Mathematics}, 41(6), 1640--1655, 2017.

\bibitem{mustafa} El-Desouky B.S., \& Mustafa A., New results on higher-order Daehee and Bernoulli numbers and polynomials, {\it Advances in Difference Equations}, 2016(32), 1--21, 2016.

\bibitem{mustafa2} El-Desouky B.S., \& Mustafa A., New results and matrix representation for Daehee and Bernoulli numbers and polynomials, {\it Applied Mathematical Sciences}, 9(73), 3593--3610, 2015.

\bibitem{atlan} El-Mikkawy M., \& Atlan F., Derivation of identities involving some special polynomials and numbers via generating functions with applications, {\it Applied Mathematics and Computation}, 220, 518--535, 2013.

\bibitem{gould} Gould H.W., Evaluation of sums of convolved powers using Stirling and Eulerian numbers, {\it The Fibonacci Quarterly}, 16(6), 488--497, 1978.

\bibitem{stack} Grinberg D., Derivative of the binomial $\binom{x}{n}$ with respect to $x$. Retrieved from: \url{https://math.stackexchange.com/q/2974977}

\bibitem{guo} Guo B.-N., Mez\H{o} I., \& Qi F., An explicit formula for Bernoulli polynomials in terms of $r$-Stirling numbers of the second kind, {\it Rocky Mountain J. Math.}, 46(6), 1919-1923, 2016.

\bibitem{jha} Jha S.K., Two new explicit formulas for the Bernoulli numbers, {\it Integers}, 20, Article \#A21, 1--5, 2020.

\bibitem{kargin} Karg{\i}n L., Cenkci M., Dil A., \& Can M., Generalized harmonic numbers via poly-Bernoulli polynomials, arXiv:2008.00284v2.

\bibitem{dil2} Karg{\i}n L., Dil A., \& Can M., Formulas for sums of powers of integers and their reciprocals, arXiv:2006.01132v1.

\bibitem{kim1} Kim D.S., \& Kim T., Daehee numbers and polynomials, {\it Applied Mathematical Sciences}, 7(120), 5969--5976, 2013.

\bibitem{kim2} Kim D.S., Kim T., Lee S.-H., \& Seo J.-J., Higher-order Daehee numbers and polynomials, {\it Int. Journal of Math. Analysis}, 8(6), 273--283, 2014.

\bibitem{lai} Laissaoui D., Bounebirat F., \& Rahmani M., On the hyper-sums of powers of integers, {\it Miskolc Math. Notes}, 18(1), 307--314, 2017.

\bibitem{liu} Liu G.-D., \& Srivastava H.M., Explicit formulas for the N\"{o}rlund polynomials $B_n^{(x)}$ and $b_n^{(x)}$, {\it Computers and Mathematics with Applications}, 51, 1377-1384, 2006.

\bibitem{mezo} Mez\H{o} I., {\it Combinatorics and Number Theory of Counting Sequences}. CRC Press. Taylor \& Francis Group. Boca Raton (FL), 2020.

\bibitem{hyper} Rim S.-H., Kim T., \& Pyo S.-S., Identities between harmonic, hyperharmonic and Daehee numbers, {\it Journal of Inequalities and Applications}, 2018(168), 1--12, 2018.

\bibitem{shirali} Shirali S., Stirling set numbers \& powers of integers, {\it At Right Angles}, pp. 26-32, March 2018.

\bibitem{todorov} Todorov P.G., On the theory of the Bernoulli polynomials and numbers, {\it Journal of Mathematical Analysis and Applications}, 104(2), 309-350, 1984.

\bibitem{wang} Wang W., Riordan arrays and harmonic number identities, {\it Computers and Mathematics with Applications}, 60, 1494-1509, 2010.



\end{thebibliography}
\end{document}